\documentclass[pandoc, label-section=modern, 12pt, twoside, leqno]{amsart}
\usepackage[top=30truemm,bottom=30truemm,left=20truemm,right=20truemm]{geometry}
\usepackage{mathtools, amssymb, amsthm, fancyhdr, tcolorbox, environ, xcolor, mathrsfs}
\usepackage[inline]{enumitem}
\numberwithin{equation}{section}

\theoremstyle{definition}
\newtheorem{theorem}{Theorem}[section]
\newtheorem{lemma}[theorem]{Lemma}
\newtheorem{proposition}[theorem]{Proposition}

\newtheorem{remark}[theorem]{Remark}

\newcommand{\R}{\mathbb{R}}

\newcommand{\Del}{\Delta}

\mathtoolsset{showonlyrefs=true}

\title[Initial--boundary problem of Zakharov type system]{Global Existence for a Zakharov type system in a domain}
\date{}
\author[N. Kobayashi and M. Ohta]{Nobutatsu Kobayashi$^\dag$ and Masahito Ohta$^\ddag$}
\thanks{\dag Department of Mathematics, Graduate School of Science, 
Tokyo University of Science,
1-3 Kagurazaka, Shinjuku-ku, Tokyo 162-8601, Japan.
Email : 1125702@ed.tus.ac.jp\\
\ddag Department of Mathematics, 
Tokyo University of Science,
1-3 Kagurazaka, Shinjuku-ku, Tokyo 162-8601, Japan.
Email : mohta@rs.tus.ac.jp
}

\begin{document}
\begin{abstract}
We study an initial-boundary value problem for a Zakharov type system in three space dimensions in a general domain. Under a smallness assumption on the initial data, we construct a unique global strong solution. The solution is obtained by a direct approach based on higher-order energy estimates and the construction of a Cauchy sequence in suitable Banach spaces, without employing compactness methods. Furthermore, we obtain estimates on the growth of higher-order Sobolev norms of solutions.


\end{abstract}

\maketitle

%Chapter 1
\section{Introduction}
In this paper, we consider the initial-boundary value problem for a Zakharov type system:
\begin{equation}
\begin{cases}
      i\partial_t u=-\Del u+vu,  &(t,x)\in\R\times\Omega,\\
      \partial_t v=-(-\Del)^{1/2} w,  &(t,x)\in\R\times\Omega,\\
      \partial_t w=(-\Del)^{1/2} v+|u|^2,  &(t,x)\in\R\times\Omega,\\
      (u(t,x),v(t,x),w(t,x))=(0,0,0),  &(t,x)\in \R\times\partial\Omega,\\
      (u(0,x),v(0,x),w(0,x))=(u_0(x),v_0(x),w_0(x)), &x\in \Omega,
\end{cases}
\label{eq:(1.1)}
\end{equation}
where $u : \mathbb{R}\times \Omega\to \mathbb{C}$, $v,w : \mathbb{R}\times \Omega\to \mathbb{R}$ are unknown functions, and $\Omega$ is a domain in $\R^3$ with smooth boundary $\partial\Omega$.

The system \eqref{eq:(1.1)} corresponds to the $L^2$-critical case of the following system:
\begin{equation}
	\left\{\begin{alignedat}{3}
      &i\partial_t u=-\Del u+vu,\\
      &\partial_t v=-(-\Del)^{1/2} w,\\
      &\partial_t w=(-\Del)^{1/2} v+(-\Del)^{\frac{1}{2}-\alpha}(|u|^2),
       \end{alignedat}\right.
       \label{eq:(1.1_2)}
\end{equation}
where $0\le\alpha\le 1$.
Differentiating the second equation in \eqref{eq:(1.1_2)} with respect to $t$ and substituting the third equation into it, we obtain the following system:
\begin{equation}
	\left\{\begin{alignedat}{3}
      &i\partial_t u+\Del u=vu,\\
      &\partial_t^2 v-\Del v=-(-\Del)^{1-\alpha}(|u|^2).
      \end{alignedat}\right.
       \label{eq:(1.1_3)}
\end{equation}

We note that when $\alpha=0$, \eqref{eq:(1.1_3)} becomes the Zakharov system:
\begin{equation}
	\left\{\begin{alignedat}{3}
      &i\partial_t u+\Del u=vu,\\
      &\partial_t^2 v-\Del v=\Del(|u|^2),
      \end{alignedat}\right.
       \label{eq_Z}
\end{equation}
which models the propagation of Langmuir waves in an ionized plasma \cite{MR2}, and when $\alpha=1$, \eqref{eq:(1.1_3)} becomes the Schr\"{o}dinger-wave system, which is a model for the interaction between a meson and a nucleon \cite{MR380160}. 

The system \eqref{eq:(1.1_3)} can be regarded as a model leading to a Hartree type equation in the whole space $\R^3$. In fact, if we consider the second equation with a parameter $c$:
\begin{align*}
\frac{1}{c^2}\partial_t^2 v-\Del v=(-\Del)^{1-\alpha}(|u|^2),
\end{align*}
by taking the subsonic limit $c \to \infty$, formally we obtain
\begin{align*}
i\partial_t u-\Del u=c_{\alpha}(|x|^{2\alpha-3}*|u|^2)u,
\end{align*}
where $c_\alpha=\frac{\Gamma(\frac{3-2\alpha}{2})}{\pi^{3/2}2^{2\alpha}\Gamma(\alpha)}$.

Although there are some results on the Cauchy problem for generalized Zakharov systems \eqref{eq:(1.1_3)} in $\mathbb{R}^3$, there are, to the best of our knowledge, no results for general domains. Beck, Pusateri, Sosoe and Wong \cite{MR3403405} proved the existence and uniqueness of global solutions $(u,v)\in C([0,\infty);H^{m+1}(\mathbb{R}^3)\times H^m(\mathbb{R}^3))$ 
to \eqref{eq:(1.1_3)} with small initial data for sufficiently large $m$ and $\alpha\in[0,1/2)$. Moreover, J. Kato and Tojo \cite{kato2024smallenergyscatteringradial} constructed unique global solutions $(u,v)\in C([0,\infty); H^1(\mathbb{R}^3)\times H^{\alpha}(\mathbb{R}^3))$ in the energy space for $\alpha\in[0,1/4]$ under the assumption that the initial data is small and radially symmetric. Both results are based on suitable versions of Strichartz estimates, which are powerful tools in the whole space $\mathbb{R}^N$. By contrast, for general domains, Strichartz estimates, the Bourgain method, and the I-method are not applicable, since Fourier analysis is no longer available. In this case, compactness arguments such as the Galerkin method, the Ascoli--Arzel\'{a} Theorem, the Banach--Alaoglu Theorem and the Rellich--Kondrachev Theorem are useful. On the other hand, the construction of solutions by compactness methods requires complex procedures.

Our aim in this paper is to construct solutions in a more direct and constructive way, independent of any compactness argument and depending only on the completeness of function spaces. We can also establish the uniqueness and continuous dependence by the same procedure. M. Hayashi and Ozawa \cite{MR3548257} proved the existence and uniqueness of local solutions to the generalized derivative nonlinear Schr\"{o}dinger equation.
They constructed the solutions by using energy estimates and proving that the sequence of approximate solutions is a Cauchy sequence in suitable Banach spaces. Furthermore, Ozawa and Tomioka \cite{MR4305957, MR4859330} proved the existence and uniqueness of global solutions to the Klein--Gordon--Schr\"{o}dinger equations:
\begin{equation}
	\left\{\begin{alignedat}{3}
      &i\partial_t u+\Del u=vu,\\
      &\partial_t^2 v-\Del v+v=|u|^2,
      \end{alignedat}\right.
       \label{eq_KGS}
\end{equation}
in four space dimensions, and the Zakharov system \eqref{eq_Z} in two space dimensions in a similar way to \cite{MR3548257}. 

Now, the Laplacian $A=-\Del$ is understood as the self-adjoint realization in the Hilbert space $L^2(\Omega)$ with dense domain $D(A)=(H^2\cap H^1_0)(\Omega)$. We denote the resolvent set of $A$ by $\rho(A)$, and assume that $0\in\rho(A)$. This assumption implies that $A$ is a positive operator with a bounded inverse $A^{-1}:L^2(\Omega)\to D(A)$. Under this framework, we can define the fractional powers of $A$ via the spectral theorem. Moreover $D(A^{1/2})=H_0^1(\Omega)$ and $\|A^{1/2}u\|_{L^2}=\|\nabla u\|_{L^2}$. Then, \eqref{eq:(1.1)} can be rewritten as follows:
\begin{equation}
	\left\{\begin{alignedat}{3}
      &i\partial_t u=Au+vu,\\
      &\partial_t v=-A^{1/2} w,\\
      &\partial_t w=A^{1/2} v+|u|^2,\\
      &(u(0),v(0),w(0))=(u_0,v_0,w_0).
       \end{alignedat}\right.
       \label{eq_1_3}
\end{equation}
The purpose of this paper is to construct a unique global strong solution to \eqref{eq_1_3}.
The system \eqref{eq_1_3} has the following conserved quantities:
\begin{align*}
	M(u)=\|u\|_{L^2}^2,\quad E(\vec{u})=\|\nabla u\|_{L^2}^2+\frac{1}{2}(\|A^{1/4}v\|_{L^2}^2+\|A^{1/4}w\|_{L^2}^2)+(v,|u|^2)_{L^2},
\end{align*}
where $\vec{u}=(u,v,w)$, and
\begin{align*}
	(u,v)_{L^2}=\mathrm{Re}\int_{\Omega}u(x)\overline{v(x)}dx.
\end{align*}
We define
\begin{align}
X_e=D(A^{1/2})\times D(A^{1/4})\times D(A^{1/4}),\quad X=D(A)\times D(A^{3/4})\times D(A^{3/4}).
\end{align}

We now state the main result of this paper.
\begin{theorem}\label{thm_1_1}
Let $\Omega$ be a domain in $\R^3$, and assume $0\in\rho(A)$ for the operator $A=-\Del$ with $D(A)=(H^2\cap H^1_0)(\Omega)$. Then, there exists a constant $C_0>0$ such that for any $\vec{u}_0=(u_0,v_0,w_0)\in X$ satisfying $\|u_0\|_{L^2}<C_0$, there exists a unique global solution $\vec{u}$ to \eqref{eq_1_3} such that
\begin{align}
\vec{u}\in C&(\mathbb{R}; X)
\cap C^1(\mathbb{R}; Y),
\end{align}
where $Y= L^2(\Omega)\times D(A^{1/4})\times D(A^{1/4})$. Moreover, the following properties hold.
\begin{itemize}
\item $\vec{u}$ depends continuously on $\vec{u}_0$ in the following sense: if $\vec{u}_0^n \to \vec{u}_0$ in $X$ and $\vec{u}_n=(u_n,v_n,w_n)$ is the corresponding solution to \eqref{eq_1_3} with $\vec{u}_n(0)=\vec{u}_0^n$, then $\vec{u}_n \to \vec{u}$ in  $C([-T,T]; H^s(\Omega)\times H^{\sigma}(\Omega)\times H^{\sigma}(\Omega))$
for any $T>0$ and $(s,\sigma)\in [0,2)\times [0,3/2)$.
\item $M(u(t))=M(u_0),\ E(\vec{u}(t))=E(\vec{u}_0)$ for all $t\in\mathbb{R}$.
\item There exists a constant $C>0$, which depends on $\|\vec{u}_0\|_X$, such that
\begin{equation}
\|\vec{u}(t)\|_{X}\le C\exp(C|t|)\quad \text{for all}\ t\in\R.
\label{est_1}
\end{equation}
\end{itemize}
\end{theorem}

\begin{remark}
The assumption $0 \in \rho(A)$ in Theorem \ref{thm_1_1} is essential for the following reasons:

(i)\ Embedding and Poincar\'{e} Inequalities: The condition $0 \in \rho(A)$ is equivalent to the strictly positive lower bound of the spectrum, $\lambda_1 = \inf \sigma(A) > 0$. Under the homogeneous Dirichlet boundary condition, this ensures the Poincar\'{e} inequality, and more generally, the fractional Poincar\'{e} inequality $\|u\|_{L^2} \le C \|A^{s/2} u\|_{L^2}$ for $s > 0$ (Lemma \ref{lem_poin}). This allows us to control the $L^2$-norm by the energy functional, which is important for establishing global-in-time estimates.

(ii)\ Equivalence of Norms: In the study of \eqref{eq_1_3}, we frequently use the equivalence between the domain $D(A^{s/2})$ of the fractional powers of $A$ and the fractional Sobolev space $\tilde{W}^{s,2}(\Omega)$ characterized by the Gagliardo seminorm. The boundedness of $A^{-1}$ ensures that the $L^2$-norm is controlled by $\dot{H}^s$ norms, allowing these spaces to be identified.

(iii)\ Scope of Domains $\Omega$: Our result is applicable to any domain $\Omega$ where the Poincar\'{e} inequality holds. This includes not only bounded domains but also unbounded domains that are bounded in at least one direction, such as $\Omega = \mathbb{R}^2 \times (0, 1)$.
\end{remark}

\begin{remark}
The constant $C_0$ in Theorem \ref{thm_1_1} is related to the best constant for the Gagliardo--Nirenberg--Sobolev inequality (see Proposition \ref{prop_3_2}).
\end{remark}

\begin{remark}
Note that the solution $u$ to the Schr\"{o}dinger equation was shown to belong to the space of weakly continuous functions $C_w(\R;D(A))$ for both \eqref{eq_Z} and \eqref{eq_KGS} in \cite{MR4305957, MR4859330}. In Theorem \ref{thm_1_1}, we prove that the solution $u$ to our system \eqref{eq_1_3} belongs to the space of strongly continuous functions $C(\R;D(A))$, which implies that it is a strong solution. We emphasize that our argument in Proposition \ref{prop_5_5} is applicable to \eqref{eq_Z} and \eqref{eq_KGS}. Consequently, by using the same argument, we can improve the regularity of the solutions in \cite{MR4305957, MR4859330} from $C_w(\R;D(A))$ to $C(\R;D(A))$.
\end{remark}

\begin{remark}
It is worth noting the exponential bound \eqref{est_1} in Theorem \ref{thm_1_1}.

\noindent(i) The estimate \eqref{est_1} is an improvement of a typical consequence of applying the Brezis--Gallouet inequality to control $\|u(t)\|_{L^\infty}$, whereby the higher-order Sobolev norms are bounded by $C\exp(Ct^2)$. In our case, by employing a more direct energy estimate and exploiting the specific structure of \eqref{eq_1_3}, we are able to establish the global existence with the exponential bound \eqref{est_1}, even in the $L^2$-critical case.

\noindent(ii) The systems \eqref{eq_Z} and \eqref{eq_KGS} correspond to the endpoint cases of \eqref{eq:(1.1_3)}. Our result shows that the same exponential bound also holds in the intermediate case.
\end{remark}

\begin{remark}
Both \eqref{eq_Z} and \eqref{eq_KGS} require Yudovich's argument based on the critical Sobolev type inequality in the $L^2$-critical cases, $N=2$ and $N=4$, respectively, while we do not require such an argument even in the $L^2$-critical case. This is due to the fact that the loss of derivative for \eqref{eq_1_3} is smaller than that for \eqref{eq_Z}, and \eqref{eq_1_3} is of a lower dimension than \eqref{eq_KGS}.
\end{remark}

This paper is organized as follows. In Section $2$, we collect the basic estimates used throughout the paper. In Section $3$, we consider approximate problems by the resolvent of the Laplacian and construct approximate solutions in the domain of the linear operator. Next, we establish the uniform bound of the approximate solutions in the energy space $D(A^{1/2})\times D(A^{1/4})\times D(A^{1/4})$ by using the mass and energy conservation laws. In Section $4$, we obtain the uniform estimates in the higher energy space $D(A)\times D(A^{3/4})\times D(A^{3/4})$ by estimating the time derivative of the first component of approximate solutions. Moreover, we prove that the sequence of approximate solutions forms a Cauchy sequence in $L^{\infty}_{\mathrm{loc}}(\R;L^2\times L^2\times L^2)$ and construct the solution of \eqref{eq_1_3} by the completeness of a function space. Finally, we confirm that the limit of the sequence of approximate solutions satisfies the equation \eqref{eq_1_3}. We emphasize that our argument does not rely on any compactness theorem.

%Chapter 2
\section{Preliminaries}
Throughout this paper, $\nabla u$ denotes the weak gradient of $u$. Note that $D(A^{1/2})=H_0^1(\Omega)$ and $\|A^{1/2}u\|_{L^2}=\|\nabla u\|_{L^2}$.
Since the nonlinear term in \eqref{eq_1_3} is not Lipschitz continuous on bounded subsets of $X=D(A)\times D(A^{3/4})\times D(A^{3/4})$, we consider the following approximate problem:
\begin{equation}
	\left\{\begin{alignedat}{4}
      &\partial_t u=-iA u-iJ_n(J_nv\cdot J_nu),\\
      &\partial_t v=-A^{1/2} w,\\
      &\partial_t w=A^{1/2} v+J_n(|J_nu|^2)\\
      &(u(0),v(0),w(0))=(J_nu_0,J_nv_0,J_nw_0),
       \end{alignedat}\right.
       \label{eq:(1.3)}
\end{equation}
where $J_n=\left(1-\frac{1}{n}\Del\right)^{-1}$ denotes the resolvent of the Laplacian.
Note that the regularized system \eqref{eq:(1.3)} preserves the Hamiltonian structure. 

We first collect several properties of $J_n$ used in the proof of the main theorem.
\begin{lemma}
    \label{lem_2_1}
Let $1<p<\infty$ and let $X$ be any of the spaces $D(A^{1/2})$, $D(A)$ or $L^p(\Omega)$ and let $X^*$ be its dual space. Then, the following properties hold:

    \noindent (1) $\langle J_n f,g\rangle_{X,X^*}=\langle f,J_n g\rangle_{X,X^*}$ for $f\in X$, $g\in X^*$.
    
    \noindent (2) $J_n$ is a bounded self-adjoint operator in $L^2(\Omega)$ with $J_n(L^2(\Omega))=D(A)$.

    \noindent (3) $\|J_n\|_{\mathcal{L}(X)}\le 1$.

    \noindent (4) 
        $J_n u\to u$ in $X$ as $n\to\infty$ for any $u\in X$.

    \noindent (5) The following inequalities hold:
    \begin{equation}
        \|\nabla J_n u\|_{L^2}\le \sqrt{n}\|u\|_{L^2},\quad \|\Del J_n u\|_{L^2}\le n\|u\|_{L^2}\quad \text{for any}\ u\in L^2(\Omega).
    \end{equation}
\end{lemma}

For the proof, see, e.g.,  \cite{MR2002047, MR3802567}.

\begin{lemma} [\cite{MR2002047, MR3802567}] 
\label{lem_2_2}
    There exists a constant $C>0$ such that
    \begin{align}
        \|(J_m-J_n)u\|_{L^2}\le C\left(\frac{1}{\sqrt{m}}+\frac{1}{\sqrt{n}}\right)\|\nabla u\|_{L^2}
    \end{align}
    for any $m,n\in\mathbb{N}, u\in D(A^{1/2})$.
\end{lemma}

Next, we recall several characterizations of fractional Sobolev spaces. Let $\tilde{W}^{s,2}(\Omega)$ be the fractional Sobolev space defined by the Gagliardo seminorm:
\begin{align}
&\tilde{W}^{s,2}(\Omega)=\left\{v\in L^2(\Omega):[v]_{s,2}<\infty\right\},\\
&[v]_{s,2}=\left(\int_{\Omega}\int_{\Omega}\frac{|v(x)-v(y)|^2}{|x-y|^{3+2s}}dxdy\right)^{1/2}.
\end{align}
It is well known that
\[
\tilde{W}^{s,2}_0(\Omega)=(L^2(\Omega),H_0^1(\Omega))_{s,2}\quad \text{for}\ 0<s<1.
\]

On the other hand, since $A$ is a positive self-adjoint operator and $D(A^{1/2})=H_0^1(\Omega)$, Theorem $4.36$ in \cite{MR3753604} yields
\[
D(A^{s/2})=(L^2(\Omega),H_0^1(\Omega))_{s,2}=[L^2(\Omega),H_0^1(\Omega)]_s\quad \text{for}\ 0<s<1.
\]

Consequently,
\[
D(A^{s/2})=\tilde{W}^{s,2}_0(\Omega).
\]

The assumption $0\in\rho(A)$ implies the following Poincar\'{e}-type inequality:

\begin{lemma}\label{lem_poin}
There exists a positive constant $C$ such that
\begin{align}
\|f\|_{L^2}\le C\|A^{s/2}f\|_{L^2}\quad \text{for all}\ f\in D(A^{s/2}).
\end{align}
\end{lemma}

\begin{proof}
Since $A$ is a positive self-adjoint operator in $L^2(\Omega)$ and $0 \in \rho(A)$ by assumption, the spectrum $\sigma(A)$ is strictly bounded away from zero. That is, $\lambda_1 \coloneqq \inf \sigma(A) > 0$. 

By the spectral theorem, for any $s > 0$ and $f \in D(A^{s/2})$,
\begin{align}
\|A^{s/2}f\|_{L^2}^2 = \int_{\lambda_1}^{\infty} \lambda^s d\|E_{\lambda}f\|_{L^2}^2 \ge \lambda_1^s \int_{\lambda_1}^{\infty} d\|E_{\lambda}f\|_{L^2}^2 = \lambda_1^s\|f\|_{L^2}^2.
\end{align}
The desired inequality follows with $C = \lambda_1^{-s/2}$.
\end{proof}

We shall also use the following interpolation inequality.
\begin{lemma}[\cite{MR2944369}]\label{lem_2_4}
Let $s\in(0,1)$. Then, there exists a constant $C>0$ such that
\begin{align}
\|f\|_{L^q}\le C\|A^{s/2}f\|_{L^2}^{\theta}\|f\|_{L^2}^{1-\theta}
\end{align}
for all $q\in[2,2_s^*]$ and $f\in D(A^{s/2})$, where $2^*_s=\displaystyle\frac{6}{3-2s}$ and $\theta=\frac{3}{s}\left(\frac{1}{2}-\frac{1}{q}\right)$.
\end{lemma}

%Chapter 3
\section{Uniform bound in the energy class $X_e$}
Throughout this section, we assume $\vec{u}_0=(u_0,v_0,w_0)\in X$. We set
\begin{equation}
L=
\begin{pmatrix*}[c]
    -iA &0 &0 \\
    0 &0 &-A^{1/2} \\
    0 &A^{1/2} &0
\end{pmatrix*},
\quad
    \mathcal{N}_n(\vec{u})=
    \begin{pmatrix*}[c]
        -iJ_n(J_n v\cdot J_n u) \\
        0 \\
        J_n(|J_n u|^2)
    \end{pmatrix*},
\end{equation}
so that \eqref{eq:(1.3)} becomes
\begin{align*}
\partial_t \vec{u}(t)=L\vec{u}(t)+\mathcal{N}_n(\vec{u}(t)).
\end{align*}

\begin{lemma}\label{lem_3_1}
The operators $\pm L$ with domains $D(\pm L)=D(A^2)\times D(A^{5/4})\times D(A^{5/4})$ are m-dissipative in $X$.
\end{lemma}
The proof of Lemma \ref{lem_3_1} is standard (see, e.g., \cite{MR1691574}).

Since $L$ is an m-dissipative operator in $X$ with dense domain, and since $\mathcal{N}_n$ is Lipschitz continuous on bounded subsets of $X$, for any $\vec{u}_0\in X$, there exists a maximal existence time $T^{*}_n>0$ and a unique local solution $\vec{u}_n\in C([0,T^*_n);D(L))\cap C^1([0,T^*_n);X)$ to \eqref{eq:(1.3)}. Note that $J_n\vec{u}_0\in D(L)$.

We define the approximate energy by
\begin{align}\label{eq_1_4}
E_n(\vec{u})=\|\nabla u\|_{L^2}^2+\frac{1}{2}\left(\|A^{1/4}v\|_{L^2}^2+\|A^{1/4}w\|_{L^2}^2\right)+(J_n v,|J_n u|^2)_{L^2}.
\end{align}
A standard calculation yields the following conservation laws of mass and energy for the approximate problem \eqref{eq:(1.3)}.
\begin{lemma}\label{lem_3_2}
The following conservation laws hold:
\begin{align}\label{lem_3_2}
M(u_n(t))=M(u_0),\ E_n(\vec{u}_n(t))=E_n(\vec{u}_0)\quad \text{for all}\ t\in[0,T^*_n).
\end{align}
\end{lemma}
Lemma \ref{lem_3_2} provides a priori estimates for the approximate solutions, which leads to the following uniform boundedness in the energy space. 

\begin{proposition}\label{prop_3_2}
There exists a constant $C_0>0$ such that if $\|u_0\|_{L^2}<C_0$, then there exists $M_1>0$ depending only on $\|\vec{u}_0\|_{X_e}$ such that
\begin{align*}
\|\vec{u}_n(t)\|_{X_e}\le M_1\quad \text{for all}\ t\in [0,T^*_n).
\end{align*}
\end{proposition}

\begin{proof}
We estimate the last term in \eqref{eq_1_4} using the Sobolev embedding and the Gagliardo--Nirenberg inequality:
\begin{align}\label{ineq_3_3}
|(J_nv_n,|J_nu_n|^2)_{L^2}|&\le \|J_nv_n\|_{L^3}\|J_nu_n\|_{L^3}^2\\
&\le C\|A^{1/4}v_n\|_{L^2}\|\nabla u_n\|_{L^2}\|u_n\|_{L^2}\\
&\le \tilde{C}\|u_0\|_{L^2}\left(\|\nabla u_n\|_{L^2}^2+\frac{1}{2}\|A^{1/4} v_n\|_{L^2}^2\right).
\end{align}
From Lemma \ref{lem_3_2} and \eqref{ineq_3_3}, we have
\begin{align}\label{ineq_3_4}
&\left(1-\tilde{C}\|u_0\|_{L^2}\right)\left(\|\nabla u_n\|_{L^2}^2+\frac{1}{2}\|A^{1/4}v_n\|_{L^2}^2\right)+\frac{1}{2}\|A^{1/4}w_n\|_{L^2}^2\\
&\le E_n(\vec{u}_n(t))=E_n(\vec{u}_n(0))\\
&\le \left(1+\tilde{C}\|u_0\|_{L^2}\right)\left(\|\nabla u_0\|_{L^2}^2+\frac{1}{2}\|A^{1/4}v_0\|_{L^2}^2\right)+\frac{1}{2}\|A^{1/4}w_0\|_{L^2}^2.
\end{align}
By \eqref{ineq_3_4} and Lemma \ref{lem_poin}, if $\|u_0\|_{L^2}<C_0=:1/\tilde{C}$, then we obtain
\begin{equation}
\sup_{n\in\mathbb{N}} \|\vec{u}_n(t)\|_{X_e}\le M_1\quad \text{for all}\ t\in[0,T_n^*).
\end{equation}
This completes the proof of Proposition \ref{prop_3_2}.
\end{proof}

%Chapter 4
\section{Uniform bound in the higher energy class $X$}
Throughout this section, we assume $\vec{u}_0\in X$. In this section, we prove the following proposition. 
\begin{proposition}\label{prop_4_1}
There exists a constant $C_1>0$ depending only on $\|\vec{u}_0\|_X$ such that
\begin{equation}
\|\vec{u}_n(t)\|_X\le C_1\exp(C(M_1)|t|)=:M_2(t)\quad \text{for all}\  t\in[0,T^*_n).
\end{equation}
\end{proposition}
We differentiate the first equation once in time instead of differentiating twice the equation in space in order to obtain the uniform $H^2$ estimate of the first component of the approximate solutions.
\begin{proof}
To estimate the $X$-norm, we define
\begin{align*}
F(\vec{u})=\|\partial_t u\|_{L^2}^2+\frac{1}{2}\left(\|A^{3/4}v\|_{L^2}^2+\|A^{3/4}w\|_{L^2}^2\right).
\end{align*}
By computing the time derivative of $F(\vec{u}_n(t))$, we have
\begin{align*}
\frac{d}{dt}F(\vec{u}_n(t))=2(\partial_t^2 u_n,\partial_t u_n)_{L^2}+(\partial_t A^{3/4}v_n,A^{3/4}v_n)_{L^2}+(\partial_t A^{3/4}w_n,A^{3/4}w_n)_{L^2}.
\end{align*}

\begin{align*}
(\partial_t^2 u_n,\partial_t u_n)_{L^2}&=(\partial_t (i\Del u_n-iJ_n(J_n v_n\cdot J_n u_n)),\partial_t u_n)_{L^2}\\
&=-(i\partial_t(J_n v_n\cdot J_n u_n),\partial_t J_n u_n)_{L^2}=-(i\partial_t J_n v_n\cdot J_n u_n,\partial_t J_n u_n)_{L^2}\\
&=-(iJ_n(\partial_t J_n v_n\cdot J_n u_n),i\Del u_n-iJ_n(J_n v_n\cdot J_n u_n))_{L^2}\\
&=-(\partial_t J_n v_n\cdot J_n u_n,\Del J_n u_n-J_n^2(J_n v_n\cdot J_n u_n))_{L^2}.
\end{align*}

\begin{align*}
(\partial_t A^{3/4}w_n,A^{3/4}w_n)_{L^2}&=-(\partial_t A^{3/4}w_n, A^{1/4}\partial_t v_n)_{L^2}=-(\partial_t A^{1/4}w_n,\partial_t A^{3/4} v_n)_{L^2}\\
&=-(A^{3/4}v_n+J_n A^{1/4}|J_n u_n|^2,\partial_t A^{3/4} v_n)_{L^2}\\
&=-(A^{3/4}v_n,\partial_t A^{3/4}v_n)_{L^2}+(\Del |J_n u_n|^2,\partial_t J_n v_n)_{L^2}.
\end{align*}
Collecting these calculations, we have
\begin{align}
\frac{d}{dt}F(\vec{u}_n(t))&=-2(\partial_t J_n v_n\cdot J_n u_n,\Del J_n u_n-J_n^2(J_n v_n\cdot J_n u_n))_{L^2}+(\partial_t J_n v_n,\Del |J_n u_n|^2)_{L^2}\label{eq_4_1} \\
&=2(\partial_t J_n v_n,|\nabla J_n u_n|^2)_{L^2}+2(\partial_t J_n v_n\cdot J_n u_n,J_n^2(J_n v_n\cdot J_n u_n))_{L^2},
\end{align}
where in the last equality we used
\begin{align*}
\Delta |J_n u_n|^2-2\mathrm{Re}(J_n\overline{u_n} \cdot \Delta J_n u_n)=2|\nabla J_n u_n|^2.
\end{align*}
We note that
\begin{align}\label{eq_4_2}
\|\partial_t u_n\|_{L^2}^2=\|\Del u_n\|_{L^2}^2&+(\nabla J_n v_n,\nabla|J_n u_n|^2)_{L^2}\\
&+2(J_n v_n,|\nabla J_n u_n|^2)_{L^2}+\|J_n(J_n v_n\cdot J_n u_n)\|_{L^2}^2.
\end{align}
By using Lemma \ref{lem_2_1}, Lemma \ref{lem_2_4} and the interpolation inequality, the inner product terms of \eqref{eq_4_2} are estimated as
\begin{align}
|(\nabla J_n v_n,&\nabla|J_n u_n|^2)_{L^2}|=|(A^{3/4}J_n v_n,A^{1/4}|J_n u_n|^2)_{L^2}|\label{eq_4_3} \le \|A^{3/4}J_n v_n\|_{L^2}\|A^{1/4}|J_n u_n|^2\|_{L^2}\\
&\le C\|A^{3/4}J_n v_n\|_{L^2}\||J_n u_n|^2\|_{H^{1/2}}\le C\|A^{3/4}v_n\|_{L^2}\||J_n u_n|^2\|_{H^2}^{1/4}\||J_n u_n|^2\|_{L^2}^{3/4}\\
&\le C\|A^{3/4}v_n\|_{L^2}\|u_n\|_{H^2}^{1/2}\|u_n\|_{L^4}^{3/2}\le C(M_1)\|u_n\|_{H^2}^{1/2}\|A^{3/4}v_n\|_{L^2}\\
&\le \left(\frac{1}{2}\|\Del u_n\|_{L^2}+C(M_1)\right)\|A^{3/4}v_n\|_{L^2}\\
&\le \frac{1}{2}\left(\frac{3}{4}\|\Del u_n\|_{L^2}^2+\frac{1}{3}\|A^{3/4}v_n\|_{L^2}^2\right)+\frac{1}{6}\|A^{3/4}v_n\|_{L^2}+C(M_1)\\
&\le \frac{3}{8}\|\Del u_n\|_{L^2}^2+\frac{1}{3}\|A^{3/4}v_n\|_{L^2}^2+C(M_1),
\end{align}

\begin{align}
|(J_n v_n,|\nabla J_n u_n|^2)_{L^2}|&\le \|v_n\|_{L^3}\|\nabla u_n\|_{L^3}^2\label{eq_4_4} \le C\|A^{1/4}v_n\|_{L^2}\|\nabla u_n\|_{L^2}\|\Del u_n\|_{L^2}\\
&\le C(M_1)\|\Del u_n\|_{L^2}\le \frac{1}{16}\|\Del u_n\|_{L^2}^2+C(M_1),
\end{align}

\begin{align}
\|J_n(J_n v_n\cdot J_n u_n)\|_{L^2}^2&\le \|v_n\|_{L^3}^2\|u_n\|_{L^6}^2\label{eq_4_5} \le C\|A^{1/4}v_n\|_{L^2}^2\|\nabla u_n\|_{L^2}^2\le C(M_1).
\end{align}
By \eqref{eq_4_2}, \eqref{eq_4_3}, \eqref{eq_4_4} and \eqref{eq_4_5}, we have
\begin{align}
\frac{1}{2}\|\Del u_n\|_{L^2}^2&+\frac{1}{6}\|A^{3/4}v_n\|_{L^2}^2+\frac{1}{2}\|A^{3/4}w_n\|_{L^2}^2\label{eq_4_6} \\
&\le F(\vec{u}_n(t))+C(M_1)\\
&\le \frac{3}{2}\|\Del u_n\|_{L^2}^2+\frac{5}{6}\|A^{3/4}v_n\|_{L^2}^2+\frac{1}{2}\|A^{3/4}w_n\|_{L^2}^2+C(M_1).
\end{align}
Moreover, from \eqref{eq_4_1} we have
\begin{align}
|(\partial_t J_n v_n,|\nabla J_n u_n|^2)_{L^2}|&=|(A^{1/2} J_n w_n,|\nabla J_n u_n|^2)_{L^2}|\label{eq_4_7} \le \|A^{1/2} w_n\|_{L^3}\|\nabla J_n u_n\|_{L^3}^2\\
&\le C\|A^{3/4}w_n\|_{L^2}\|\Del u_n\|_{L^2}\|\nabla u_n\|_{L^2}\le C(M_1)\|\Del u_n\|_{L^2}\|A^{3/4}w_n\|_{L^2},
\end{align}

\begin{align}
|(\partial_t J_n v_n&\cdot J_n u_n,J_n^2(J_n v_n\cdot J_n u_n))_{L^2}|\le \|A^{1/2} J_n w_n\|_{L^3}\|J_n\overline{u_n}J_n^2(J_n v_n\cdot J_n u_n)\|_{L^{3/2}}\label{eq_4_8} \\
&\le C\|A^{3/4}w_n\|_{L^2}\|u_n\|_{L^6}\|J_n v_n\cdot J_n u_n\|_{L^2}\le C\|A^{3/4}w_n\|_{L^2}\|\nabla u_n\|_{L^2}\|u_n\|_{L^6}\|v_n\|_{L^3}\\
&\le C\|A^{3/4}w_n\|_{L^2}\|\nabla u_n\|_{L^2}^2\|A^{1/4}v_n\|_{L^2}\le C(M_1)\|A^{3/4}w_n\|_{L^2}.
\end{align}
By \eqref{eq_4_1}, \eqref{eq_4_6}, \eqref{eq_4_7} and \eqref{eq_4_8}, we have
\begin{align*}
\left|\frac{d}{dt}F(\vec{u}_n(t))\right|&\le C(M_1)(1+\|\Del u_n\|_{L^2})\|A^{3/4}w_n\|_{L^2}\le C(M_1)(1+F(\vec{u}_n(t))).
\end{align*}
Applying Gronwall's inequality, there exists $C_1=C_1(\|\vec{u}_0\|_{X})>0$ such that
\begin{align}
F(\vec{u}_n(t))&\le C(M_1)(1+F(\vec{u}_0))\exp{(C(M_1)|t|)}\label{eq_4_9}\le C_1\exp{(C(M_1)|t|)}=:M_2(t).
\end{align}
\end{proof}
Proposition \ref{prop_4_1} implies that the $X$-norm is estimated a priori. Consequently, the local solutions to \eqref{eq:(1.3)} can be extended globally in time as $\vec{u}_n\in C(\mathbb{R};D(L))\cap C^1(\mathbb{R};X)$.

%Chapter 5
\section{Proof of Theorem \ref{thm_1_1}}
In this section, let $T>0$ and $M_2=M_2(T)$. By using the uniform estimate in Proposition \ref{prop_4_1} obtained in the previous section, we prove that the sequence of approximate solutions $(\vec{u}_n)_{n\in\mathbb{N}}$ converges in $L^\infty((-T,T);L^2(\Omega)\times L^2(\Omega)\times L^2(\Omega))$.

A straightforward calculation shows
 \begin{align}
 \frac{1}{2}\frac{d}{dt}\|u_m(t)-u_n(t)\|_{L^2}^2&=(\partial_t (u_m-u_n),u_m-u_n)_{L^2}\label{eq_5_1} \\
 &=-(i(J_m(J_m v_m\cdot J_m u_m)-J_n(J_n v_n\cdot J_n u_n)),u_m-u_n)_{L^2}\\
 &=-(iJ_m(v_m-v_n)\cdot J_m u_m,J_m(u_m-u_n))_{L^2}\\
 &\quad -(i(J_m-J_n)v_n\cdot J_m u_m,J_m(u_m-u_n))_{L^2}\\
 &\quad -(iJ_n v_n\cdot (J_m-J_n)u_n,J_m(u_m-u_n))_{L^2}\\
 &\quad -(i(J_m-J_n)(J_n v_n\cdot J_n u_n),u_m-u_n)_{L^2}\\
 &=:I_1+I_2+I_3+I_4.
\end{align}
By Lemma \ref{lem_2_1} and Lemma \ref{lem_2_2}, the terms in the right-hand side of \eqref{eq_5_1} are estimated as
\begin{align}
&
\begin{aligned}
I_1&\le \|J_m u_m\|_{L^\infty}\|u_m-u_n\|_{L^2}\|v_m-v_n\|_{L^2}\label{eq_5_2} \le C(M_2)\|u_m-u_n\|_{L^2}\|v_m-v_n\|_{L^2},
\end{aligned}
\\
&
\begin{aligned}
I_2&\le \|(J_n-J_m)v_n\|_{L^2}\|J_m u_m\|_{L^\infty}\|u_m-u_n\|_{L^2}\label{eq_5_3} \le C(M_2)\left(\frac{1}{\sqrt{m}}+\frac{1}{\sqrt{n}}\right),
\end{aligned}
\end{align}
Similarly, we obtain
\begin{align}
I_3,I_4\le C(M_2)\left(\frac{1}{\sqrt{m}}+\frac{1}{\sqrt{n}}\right)\label{eq_5_4}.
\end{align}
By using the system \eqref{eq_1_3}, we have
\begin{align}
&\frac{1}{2}\frac{d}{dt}(\|v_m(t)-v_n(t)\|_{L^2}^2+\|w_m(t)-w_n(t)\|_{L^2}^2)\label{eq_5_5}\\
&=(\partial_t(v_m-v_n),v_m-v_n)_{L^2}+(\partial_t(w_m-w_n),w_m-w_n)_{L^2}\\
%&=(-(-\Del)^{1/2}(v_m-v_n),v_m-v_n)_{L^2}+((-\Del)^{1/2}(w_m-w_n)+J_m|J_m u_m|^2-J_n|J_n u_n|^2,w_m-w_n)_{L^2}\\
&=(J_m|J_m u_m|^2-J_n|J_n u_n|^2,w_m-w_n)_{L^2}\\
&=(J_m u_m\cdot \overline{J_m(u_m-u_n)},J_m(w_m-w_n))_{L^2}\\
&\quad +(J_m u_m\cdot \overline{(J_m-J_n)u_n},J_m(w_m-w_n))_{L^2}\\
&\quad +(J_m(u_m-u_n)\cdot \overline{J_n u_n},J_m(w_m-w_n))_{L^2}\\
&\quad +((J_m-J_n)u_n\cdot \overline{J_n u_n},J_m(w_m-w_n))_{L^2}\\
&\quad +((J_m-J_n)|J_n u_n|^2,w_m-w_n)_{L^2}\\
&=II_1+II_2+II_3+II_4+II_5.
\end{align}
By the H\"{o}lder and Sobolev inequalities and Lemma \ref{lem_2_2},  the terms in the right-hand side of \eqref{eq_5_5} are estimated as
\begin{align}
&
\begin{aligned}
II_1,II_3&\le \|J_m u_m\cdot \overline{J_m(u_m-u_n)}\|_{L^2}\|J_m(w_m-w_n)\|_{L^2}\label{eq_5_6}\\
&\le \|J_m u_m\|_{L^\infty}\|J_m(u_m-u_n)\|_{L^2}\|w_m-w_n\|_{L^2}\\
&\le C(M_2)\|u_m-u_n\|_{L^2}\|w_m-w_n\|_{L^2},
\end{aligned}
\\
&
\begin{aligned}
II_2,II_4&\le \|J_m u_m\cdot \overline{(J_m-J_n)u_n}\|_{L^2}\|J_m(w_m-w_n)\|_{L^2}\label{eq_5_7}\\
&\le \|J_m u_m\|_{L^\infty}\|(J_m-J_n)u_n\|_{L^2}\|w_m-w_n\|_{L^2}\le C(M_2)\left(\frac{1}{\sqrt{m}}+\frac{1}{\sqrt{n}}\right),
\end{aligned}
\\
&
\begin{aligned}
II_5&\le\|(J_m-J_n)|J_n u_n|^2\|_{L^2}\|w_m-w_n\|_{L^2}\label{eq_5_8}\\
&\le C\left(\frac{1}{\sqrt{m}}+\frac{1}{\sqrt{n}}\right)\|\nabla|J_n u_n|^2\|_{L^2}\|w_m-w_n\|_{L^2}\le C(M_2)\left(\frac{1}{\sqrt{m}}+\frac{1}{\sqrt{n}}\right).
\end{aligned}
\end{align}

By \eqref{eq_5_2}, \eqref{eq_5_3}, \eqref{eq_5_4}, \eqref{eq_5_6}, \eqref{eq_5_7}, \eqref{eq_5_8} and Young's inequality, we have
\begin{align}
\frac{d}{dt}\|\vec{u}_m(t)-\vec{u}_n(t)\|_{(L^2)^3}^2
\le C(M_2)\|\vec{u}_m(t)-\vec{u}_n(t)\|_{(L^2)^3}^2+C(M_2)\left(\frac{1}{\sqrt{m}}+\frac{1}{\sqrt{n}}\right).
\end{align}

From Gronwall's lemma and Lemma \ref{lem_2_2}, we have $\sup_{t\in(-T,T)}\|\vec{u}_m(t)-\vec{u}_n(t)\|_{(L^2)^3}^2\to 0$.
Since $(\vec{u}_n)_{n\in\mathbb{N}}$ is a Cauchy sequence in $C([-T,T];L^2(\Omega)\times L^2(\Omega)\times L^2(\Omega))$, there exists $\vec{u}=(u,v,w)$ of the sequence $(\vec{u}_n)_{n\in\mathbb{N}}$. By interpolation estimates and the uniform bound in $X$ obtained in Proposition \ref{prop_4_1}, we obtain $\vec{u}_n\to \vec{u}$ in $C([-T,T];H^s(\Omega)\times H^{\sigma}(\Omega)\times H^{\sigma}(\Omega))$ for any $(s, \sigma)\in(0,2)\times (0,3/2)$. From this and Lemma \ref{lem_3_2}, we have the following lemma.

\begin{lemma}\label{lem_5_1}
The following conservation laws hold:
\begin{align*}
M(u(t))=M(u_0),\quad E(\vec{u}(t))=E(\vec{u}_0)\quad \text{for all}\ t\in\mathbb{R}.
\end{align*}
\end{lemma}
\begin{proposition}\label{prop_5_2}
There exists a constant $C>0$ depending on $\|\vec{u}_0\|_{X}$ such that
\begin{align}
\|\vec{u}(t)\|_{X}\le C\exp(C|t|)\quad \text{for all}\ t\in\R.
\end{align}
\end{proposition}
\begin{proof}
Since $(u_n(t))_{n\in\mathbb{N}}$ is a bounded sequence in $D(A)$, there exists a subsequence $(u_{n_j}(t))_{j}\subset (u_n(t))_{n}$ and $u^*(t)\in D(A)$ such that
\begin{align}
&u_{n_j}(t)\rightharpoonup u^*(t)\quad \text{weakly in}\ D(A),\\
&\|u^*(t)\|_{H^2}\le \liminf_{j\to\infty}\|u_{n_j}(t)\|_{H^2}\le M_2(t),
\end{align}
where the weak limit $u^*(t)$ coincides with $u(t)$, since $(u_{n_j})_{j}$ converges to $u$ in $C([-T,T];L^2(\Omega))$. Similarly, we obtain the desired growth estimate with respect to $v, w$.
\end{proof}

\begin{proposition}\label{prop_5_3}
The function $\vec{u}$ obtained as the limit of the approximate solutions $\vec{u}_n$ satisfies the system \eqref{eq_1_3} and
\begin{align}
&u\in (L^\infty_{\mathrm{loc}}\cap C_w)(\mathbb{R};D(A))\cap C(\mathbb{R};H^1_0(\Omega))\cap C^1(\mathbb{R};H^{-1}(\Omega)),\\
&\partial_t u\in (L^\infty_{\mathrm{loc}}\cap C_w)(\mathbb{R};L^2(\Omega)),\ (v,w) \in C(\mathbb{R};D(A^{3/4})\times D(A^{3/4})).
\end{align}
\end{proposition}

\begin{proof}
Let $\mathcal{T}(t)$ be the group generated by
\begin{equation}
\begin{pmatrix}
    0 & -A^{1/2} \\
    A^{1/2} & 0
\end{pmatrix},
\end{equation}
and define
\begin{equation}
N_n(\vec{u})=
\begin{pmatrix}
    0 \\
    J_n|J_n u|^2
\end{pmatrix},\quad
N(\vec{u})=
\begin{pmatrix}
    0 \\
    |u|^2
\end{pmatrix}.
\end{equation}
We set $\vec{V}_n(t)=(v_n(t),w_n(t))$. Then, $\vec{V}_n(t)$ satisfies the following integral equation
\begin{equation}
    \vec{V}_n(t)=\mathcal{T}(t)\vec{V}_n(0)+\int_{0}^{t}\mathcal{T}(t-s)N_n(\vec{u}_n(s))ds.
\end{equation}
Since the group $\mathcal{T}(t)$ is a unitary group in $L^2(\Omega)\times L^2(\Omega)$ and $\sup_{t\in[-T,T]}\|J_n|J_n u_n(t)|^2-|u(t)|^2\|_{L^2}\to 0$, it follows that $\vec{V}(t)=(v(t),w(t))$ satisfies

\begin{equation}
    \vec{V}(t)=\mathcal{T}(t)\vec{V}(0)+\int_{0}^{t}\mathcal{T}(t-s)N(\vec{u}(s))ds.
\end{equation}
Since we have $|u|^2\in L_{\mathrm{loc}}^{\infty}(\R;D(A^{3/4}))$ and the initial data $\vec{V}(0)\in D(A^{3/4})\times D(A^{3/4})$, we obtain $\vec{V}\in C(\R;D(A^{3/4})\times D(A^{3/4}))$.
\end{proof}

\begin{proposition}
The solutions of \eqref{eq_1_3} are unique with respect to the initial data.
\end{proposition}

\begin{proof}
It suffices to check $\vec{u}_1=\vec{u}_2$ where $\vec{u}_j=(u_j,v_j,w_j)\ (j=1,2)$ is a solution to \eqref{eq_1_3}. The uniqueness is proved in the same way as in the proof of the existence (see \cite{MR4305957}). 
\end{proof}

\begin{proposition}\label{prop_5_5}
The solution $u$ to \eqref{eq_1_3} is a strong solution in the sense that $u\in C(\R;D(A))$.
\end{proposition}

\begin{proof}
By the uniqueness, it suffices to prove the continuity at $t=0$. By the result in Proposition \ref{prop_5_3}, we have
\begin{align}
\partial_t u(t)\rightharpoonup \partial_t u(0)\quad \text{weakly in}\ L^2(\Omega),\quad \|\partial_t u(0)\|_{L^2}\le \liminf_{t\to 0}\|\partial_t u(t)\|_{L^2}.
\end{align}
The higher energy estimate \eqref{eq_4_9} implies that
\begin{align}
\frac{d}{dt}F(\vec{u}_n(t))\le C(M_2).\label{eq_5_9}
\end{align}
By \eqref{eq_5_9} and Proposition \ref{prop_5_3}, we obtain
\begin{align}
F(\vec{u}(t))\le \liminf_{n\to\infty}F(\vec{u}_n(t))
\le \liminf_{n\to\infty}F(\vec{u}_n(0))+C|t|\le F(\vec{u}(0))+C|t|,
\end{align}
where in the last inequality we have used the first equation in \eqref{eq_1_3} and Lemma \ref{lem_2_1}. Thus, taking $\limsup_{t\to 0}$ yields the desired result.
\end{proof}
The continuous dependence is verified by the same argument as in the proof of the uniqueness. This completes the proof of Theorem \ref{thm_1_1}.

\section*{Acknowledgments}
The authors would like to express their sincere gratitude to Professor Mamoru Okamoto and Professor Toru Nogayama for their valuable comments.

\bibliographystyle{siam}
\bibliography{references}

\end{document}